\documentclass[11pt,a4paper]{amsart}
\usepackage{graphics}
\usepackage{color}
\usepackage{epic}
\usepackage[latin1]{inputenc}
\newtheorem{theo}{\bf Theorem}[section]
\newtheorem{propo}[theo]{\bf Proposition}
\newtheorem{lemma}[theo]{\bf Lemma}

\theoremstyle{remark}

\theoremstyle{definition}
\newtheorem{definition}[theo]{\bf Definition}

\begin{document}
\title {\bf Words with intervening neighbours in infinite Coxeter groups are reduced}
\author{{\sc Henrik Eriksson} and {\sc Kimmo Eriksson}}
\thispagestyle{empty}   

\begin{abstract}
Consider a graph with vertex set $S$. A word in the alphabet $S$ has
the intervening neighbours property if any two occurrences of
the same letter are separated by all its graph neighbours.
For a Coxeter graph, words represent group elements. Speyer recently proved 
that words with the intervening neighbours property are irreducible if the 
group is infinite and irreducible. We present a new and shorter proof using 
the root automaton for recognition of irreducible words.
\end{abstract}
\maketitle
\section{Words with intervening neighbours}
\noindent
Let $G$ be the Coxeter graph of a Coxeter group with generators 
$S$. Consider a word $w$ in the alphabet $S$. 
\begin{definition}
A word has the 
{\em intervening neighbours} property if any two occurrences of
the same letter are separated by all its graph neighbours.
\end{definition}
\noindent
In the example below, $s_0 s_1 s_0 s_2$ has this property, but
$s_0 s_1 s_0 s_2 s_1$ lacks it, since the two occurrences of $s_1$
are not separated by the neighbour $s_3$.
 
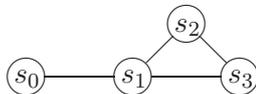
\begin{figure}[h]
 \begin{picture}(80,20)(-20,0)
    \put(-40,0){\circle{14}}
    \put(-44.5,-2.5){$s_0$}
    \put(-33,0){\line(1,0){26}}
    \put(0,0){\circle{14}}
    \put(-4.5,-2.5){$s_1$}
    \put(7,0){\line(1,0){26}}
    \put(40,0){\circle{14}}
    \put(35.5,-2.5){$s_3$}
    \put(5,5){\line(1,1){10}}
    \put(20,20){\circle{14}}
    \put(15.5,17){$s_2$}
    \put(35,5){\line(-1,1){10}}
  \end{picture}
\end{figure}

\noindent
David Speyer \cite{Speyer} recently proved  the following result.
\begin{theo}
For an infinite irreducible Coxeter group, all words with the
intervening neighbours property are irreducible.
\end{theo}
\noindent
In this note, we will demonstrate how the proof of this general result 
can be reduced to checking the property for just the affine Coxeter 
groups and just a small subset of words, for which verification of the 
property is straightforward. Our tool will be the finite automaton for 
recognition of irreducible words, invented in \cite{H}. 

\section{The root automaton}
\noindent
For any group given by generators and relations, a word $w$ in the generators
is called {\em irreducible} if it is the shortest word for that group
element. In general, recognizing irreducible words is an undecidable problem.
For a Coxeter group, however, a finite recognizing automaton exists \cite{BH}.
We will here use the concrete {\em root automaton} developed by
H.~Eriksson (for details, see \cite{BB}). 

In brief, a root in a Coxeter group can be represented as a vector 
of numbers, one for each vertex of the Coxeter graph. Let $m_{xy}\ge 3$ denote
the label of the edge between two neighbouring vertices $x$ and $y$ in the Coxeter graph. 
The set of roots is generated from the unit vectors by sequences of "reflections" 
indexed by the vertices. The reflection corresponding to a vertex $x$ changes 
only the $x$-component of the vector; to obtain the new $x$-component, change 
the sign of its previous value and for each neighbour $y$ of $x$ add the 
$y$-component value weighted by $2\cos (\pi/m_{xy})\ge 1$. 

We partially order roots by componentwise $\le$. It is a fundamental
fact in Coxeter theory that the nonzero values in a root are either all positive or
else all negative, so the poset has a negative side and a positive side.

Figure \ref{automaton} illustrates the root poset for the affine
Coxeter group $\tilde A_2$, for which the Coxeter graph is a cycle with three
vertices, say $a,b,c$, and all edge labels equal to 3.

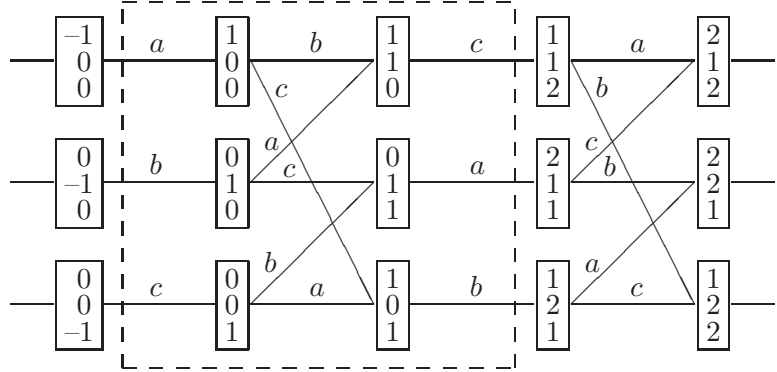
\begin{figure}\label{automaton}
\begin{picture}(300,140)(-30,-10)
\put(-17,14){\line(1,0){16}}
\put(  0,0){\fbox{\shortstack{\ 0\\ \ 0\\ --1}}}\put(18,14){\line(1,0){42}}
\put( 60,0){\fbox{\shortstack{0\\ 0\\ 1}}}\put(73,14){\line(1,0){46}}
\put(120,0){\fbox{\shortstack{1\\ 0\\ 1}}}\put(133,14){\line(1,0){46}}
\put(180,0){\fbox{\shortstack{1\\ 2\\ 1}}}\put(193,14){\line(1,0){46}}
\put(240,0){\fbox{\shortstack{1\\ 2\\ 2}}}\put(253,14){\line(1,0){16}}
\put(-17,60){\line(1,0){16}}
\put(  0,46){\fbox{\shortstack{\ 0\\ --1\\ \ 0}}}\put(18,60){\line(1,0){42}}
\put( 60,46){\fbox{\shortstack{0\\ 1\\ 0}}}\put(73,60){\line(1,0){46}}
\put(120,46){\fbox{\shortstack{0\\ 1\\ 1}}}\put(133,60){\line(1,0){46}}
\put(180,46){\fbox{\shortstack{2\\ 1\\ 1}}}\put(193,60){\line(1,0){46}}
\put(240,46){\fbox{\shortstack{2\\ 2\\ 1}}}\put(253,60){\line(1,0){16}}
\put(  0,92){\fbox{\shortstack{--1\\ \ 0\\ \ 0}}}\put(18,106){\line(1,0){42}}
\put( 60,92){\fbox{\shortstack{1\\ 0\\ 0}}}\put(73,106){\line(1,0){46}}
\put(120,92){\fbox{\shortstack{1\\ 1\\ 0}}}\put(133,106){\line(1,0){46}}
\put(180,92){\fbox{\shortstack{1\\ 1\\ 2}}}\put(193,106){\line(1,0){46}}
\put(240,92){\fbox{\shortstack{2\\ 1\\ 2}}}\put(253,106){\line(1,0){16}}
\put(-17,106){\line(1,0){16}}
 \put(73,14){\line(1,1){46}}
 \put(193,14){\line(1,1){46}}
 \put(73,60){\line(1,1){46}}
 \put(193,60){\line(1,1){46}}
 \put(73,106){\line(1,-2){46}}
 \put(193,106){\line(1,-2){46}}
\put(35,17){$c$}\put(95,17){$a$}\put(155,17){$b$}\put(215,17){$c$}
\put(35,63){$b$}\put(85,62){$c$}\put(155,63){$a$}\put(205,62.5){$b$}
\put(35,109){$a$}\put(95,109){$b$}\put(155,109){$c$}\put(215,109){$a$}
\put(78,26){$b$}\put(198,26){$a$}
\put(78,72){$a$}\put(198,72){$c$}
\put(82,92){$c$}\put(202,92){$b$}
\put(25,-10){\dashbox{5}(147,138){}}
\end{picture}
\caption{The infinite root poset of $\tilde A_2$, with the small roots indicated by the dashed box.}
\end{figure}

In order to interpret the root poset as an automaton, we let words represent
paths in the poset: The path starts at the unit root corresponding to the first 
letter of the word. For each subsequent letter, follow the the corresponding
edge in the poset (i.e., perform the corresponding reflection). 

All paths start on the positive side (because unit roots are positive). 
Paths cross over to the negative side if and only if the corresponding word may be 
shortened (reduced) by deletion of its first and the letter where the crossing occurs \cite{BB, H}. 
For example, in Figure \ref{automaton} the word $acac$ gives a path that ends by crossing
to the negative side, so we can delete the first and last letters and
obtain $ca$ as a reduced word for the same group element.
Thus, we have a deterministic automaton that detects reductions involving the first letter.

Infinite groups have an infinite number of roots, but actually the automaton only needs states 
for the finitely many {\em small roots}, 
defined as the roots that can be reached from the unit roots
without taking any step that changes a component by 2 or more. 
For example, the word $abc$ takes the automaton in Figure \ref{automaton} through three states:
$$\framebox{1 0 0}\rightarrow\framebox{1 1 0}\rightarrow\framebox{1 1 2} .$$ 
The last move increased the $c$-component by 2, so $\framebox{1 1 2}$ is not a
small root, and in fact there are just six small roots in our example. 

An automaton for recognition of irreducible words needs the small roots only, for when the current 
state has left the small roots (and reached a "big root"), it will not return to the small roots 
as long as the word is irreducible \cite{BB,H}. But such a path would have to return to
the small roots strictly before crossing over to the negative side, and then
we already know that the word allows some reduction not involving the first letter.
For example,  in deciding reducibility of a word in  $\tilde A_2$ 
starting with $abc\cdots$, we may as well delete the first letter and look at
$bc\cdots$.

\section{Intervening neighbours words in the infinite Coxeter groups }
\noindent
It is now immediate that if a word  in  $\tilde A_2$  
has the intervening neighbours property, then it is irreducible: The path taken by the
first three letters of such words always leaves the small roots, so we can simply iterate
deletion of the first letter until we are left with a two-letter word, evidently reduced.

This gives a general technique to prove irreducibility of a word: It is enough
to prove that, starting anywhere in the word, the root automaton will always reach a
big root before it reaches a negative root. For intervening neighbours words in the affine 
groups $\tilde A_n$, Speyer's theorem follows from the following simple argument.

To begin with, the 1 in the initial state will propagate to neighbour vertices, then to their
neighbours, and so on. Because of the intervening neighbours property, these 1-values will
not change before the neighbour 0-value has been raised. The last 0-value to be raised has 
two neighbour 1-values, so the automaton reaches a big root before reaching the negative side.   

\begin{propo}\label{pr:A}
For all affine groups of type  $\tilde A_n$, all words with the
intervening neighbours property are irreducible.
\end{propo}

A Coxeter graph defines an infinite group if and only if it has a subgraph
like one of the graphs of the affine groups, possibly with increased edge 
values (see \cite{BB}). Our next two propositions
establish that it is sufficient to prove the main theorem
for the graphs of the affine groups, depicted in Table \ref{coxfin}.

\begin{table}[hb]\centering\addtolength{\unitlength}{0.3pt}
\begin{tabular}{|lc @{\hspace{10pt}}|lc@{\hspace{10pt}}|lc@{\hspace{10pt}}|}
\hline
\raisebox{12pt}{$\tilde A_n$} &
\begin{picture}(30,25)(10,-5)
    \put(10,0){\circle{4}}     \put(25,10){\circle{4}}
    \put(20,-0.3){$\ldots$}    
    \put(40,0){\circle{4}}     \put(32,0){\line(1,0){6}} 
    \put(12,0){\line(1,0){6}}  
    \put(11.5,1){\line(3,2){12}}\put(38.5,1){\line(-3,2){12}}
\end{picture} &

\raisebox{12pt}{$\tilde B_n$} &
\begin{picture}(50,25)(0,-15)
    \put( 0,0){\circle{4}}     \put(10,0){\circle{4}}
    \put(20,-0.3){$\ldots$}    \put(4,1){$\scriptstyle 4$} 
    \put(40,0){\circle{4}}     \put(32,0){\line(1,0){6}} 
    \put( 2,0){\line(1,0){6}}  \put(12,0){\line(1,0){6}}
    \put(40,-2){\line(0,-1){6}}\put(40,-10){\circle{4}}  
    \put(50,0){\circle{4}}     \put(42,0){\line(1,0){6}} 
\end{picture} &

\raisebox{12pt}{$\tilde C_n$} &
\begin{picture}(50,25)(0,-10)
    \put( 0,0){\circle{4}}     \put(10,0){\circle{4}}
    \put(20,-0.3){$\ldots$}    \put( 4,1){$\scriptstyle 4$} 
    \put(40,0){\circle{4}}     \put(44,1){$\scriptstyle 4$} 
    \put(50,0){\circle{4}}     \put(42,0){\line(1,0){6}} 
    \put( 2,0){\line(1,0){6}}  \put(12,0){\line(1,0){6}}
    \put(32,0){\line(1,0){6}} 
\end{picture} 
\\ \hline

\raisebox{12pt}{$\tilde D_n$}  &
\begin{picture}(50,25)(0,-15)
    \put( 0,0){\circle{4}}     \put(10,0){\circle{4}}
    \put(20,-0.3){$\ldots$}      \put(40,-10){\circle{4}}
    \put(40,0){\circle{4}}     \put(10,-10){\circle{4}}
    \put(50,0){\circle{4}}     \put(42,0){\line(1,0){6}} 
    \put( 2,0){\line(1,0){6}}  \put(12,0){\line(1,0){6}}
    \put(32,0){\line(1,0){6}} 
    \put(10,-2){\line(0,-1){6}} \put(40,-2){\line(0,-1){6}} 
\end{picture} &

\raisebox{12pt}{$\tilde E_6$} &
\begin{picture}(40,25)(0,-21)
    \put( 0,0){\circle{4}}     \put(10,0){\circle{4}}
    \put(20,0){\circle{4}}     \put(30,0){\circle{4}}
    \put(40,0){\circle{4}}     \put(20,-10){\circle{4}}
    \put( 2,0){\line(1,0){6}}  \put(12,0){\line(1,0){6}}
    \put(22,0){\line(1,0){6}}  \put(32,0){\line(1,0){6}} 
    \put(20,-2){\line(0,-1){6}} 
    \put(20,-20){\circle{4}} \put(20,-12){\line(0,-1){6}} 
\end{picture} &

\raisebox{12pt}{$\tilde E_7$} &
\begin{picture}(60,25)(-10,-15)
    \put(-10,0){\circle{4}} \put(-8,0){\line(1,0){6}}
    \put( 0,0){\circle{4}}     \put(10,0){\circle{4}}
    \put(20,0){\circle{4}}     \put(30,0){\circle{4}}
    \put(40,0){\circle{4}}     \put(20,-10){\circle{4}}
    \put(50,0){\circle{4}}     \put(42,0){\line(1,0){6}} 
    \put( 2,0){\line(1,0){6}}  \put(12,0){\line(1,0){6}}
    \put(22,0){\line(1,0){6}}  \put(32,0){\line(1,0){6}} 
    \put(20,-2){\line(0,-1){6}} 
\end{picture}
\\ \hline
\raisebox{12pt}{$\tilde E_8$} &
\begin{picture}(70,25)(0,-15)
    \put( 0,0){\circle{4}}     \put(10,0){\circle{4}}
    \put(20,0){\circle{4}}     \put(30,0){\circle{4}}
    \put(40,0){\circle{4}}     \put(20,-10){\circle{4}}
    \put(50,0){\circle{4}}     \put(42,0){\line(1,0){6}} 
    \put(60,0){\circle{4}}     \put(52,0){\line(1,0){6}} 
    \put(70,0){\circle{4}}     \put(62,0){\line(1,0){6}} 
    \put( 2,0){\line(1,0){6}}  \put(12,0){\line(1,0){6}}
    \put(22,0){\line(1,0){6}}  \put(32,0){\line(1,0){6}} 
    \put(20,-2){\line(0,-1){6}} 
\end{picture} &

\raisebox{12pt}{$\tilde F_4$} &
\begin{picture}(40,25)(0,-10)
    \put( 0,0){\circle{4}}     \put(10,0){\circle{4}}
    \put(20,0){\circle{4}}     \put(30,0){\circle{4}}
    \put(40,0){\circle{4}}     \put(32,0){\line(1,0){6}}  
    \put( 2,0){\line(1,0){6}}  \put(12,0){\line(1,0){6}}
    \put(22,0){\line(1,0){6}}  \put(14,1){$\scriptstyle 4$} 
\end{picture} &

\raisebox{12pt}{$\tilde G_2$} &
\begin{picture}(20,25)(0,-10)
    \put( 0,0){\circle{4}}     \put(10,0){\circle{4}}
    \put( 2,0){\line(1,0){6}}  \put(4,1){$\scriptstyle 6$} 
    \put(20,0){\circle{4}}    \put(12,0){\line(1,0){6}} 
\end{picture}
\\ \hline 
\end{tabular}
\medskip
\caption{Coxeter graphs for all affine Coxeter groups}
\label{coxfin}
\end{table}

\begin{propo}
If a Coxeter graph $G$ has the property that all words with
intervening neighbours are irreducible, this property also holds for
the graph $G'$ obtained by extending $G$ with a vertex $s'$ and
an edge 
\begin{picture}(20,5)(-2,0)
\put(4,3){\line(1,0){12}}
\put(-2,0){s}
\put(16.5,0){$s^\prime$}
\end{picture}
.
\end{propo}
\begin{proof}
Consider an intervening neighbours word $w'$ in the
extended vertex set. We may write $w'=w_0s'w_1s'w_2s'\cdots$, where the $w_i$ are words 
in the $G$-vertices. By assumption, the word $w_0w_1w_2\cdots$ is reduced, so it would
take the root automaton through a series of positive roots. Now switch to the word
$w_0s'w_1w_2\cdots$. The difference comes when we play the $s$ that necessarily is in 
$w_1$, for now the extra vertex $s'$ may make a positive contribution, say $x$. So $x$ is
added to the $s$-component and this effect propagates additively through $w_1w_2\cdots$.
Since this is an intervening neighbours word, all additives will be positive or zero.

The same argument holds for the other occurrences of $s'$ in $w'$, so the states of the
automaton will certainly stay positive. 
\end{proof}

\begin{propo}
If a Coxeter graph $G$ has the property that all words with
intervening neighbours are irreducible, this property also holds for
the graph obtained by increasing an edge label value 
\begin{picture}(23,9)(-2,0)
\put(3.5,3){\line(1,0){13}}
\put(-2,0){s}
\put(7,3.5){\tiny k}
\put(16.5,0){t}
\end{picture}
to 
\begin{picture}(23,12)(-2,0)
\put(3.5,3){\line(1,0){13}}
\put(-2,0){s}
\put(6,3.5){\tiny m}
\put(16.5,0){t}
\end{picture}
, where $k<m$. 
\end{propo}
\begin{proof}
Let $w$ be an intervening neighbours word, so it is reduced over $G$
and takes the automaton through positive roots. Now use the edge label
$m$ the first time a value is transported along this edge, from $s$ to
$t$, say. The result is a raise of the $t$-value and this effect propagates
additively as a positive contribution, when the rest of the word is played.
The same argument holds for all later uses of the $m$-label.
\end{proof}

\noindent
It remains for us to prove the main theorem for the graphs in Table \ref{coxfin}.
We have already covered the cyclic case in Proposition \ref{pr:A}.
All the other eight graphs are trees. 
For each of these graphs and each start vertex
$s$ in that graph, we will define the infinite intervening neighbours 
{\em bicoloured word} in the following way. Colour all neighbours of
$s$ black, colour the neighbours' neighbours white etc, so that black
and white vertices alternate. The bicoloured word starts with $s$ followed
by all blacks (in any order, they commute!), then all whites, then all
blacks etc. The intervening neighbours property is obvious.

The automaton action is so simple for the bicoloured word that the
calculations can be performed mentally. For example, in  $\tilde E_6$
with the center vertex as $s$, we start with a 1 on $s$, then we get
1 on its neighbours, then 2 on $s$, then 2 on its neighbours and
finally 4 on $s$, a raise by 2. Other start vertices also quickly
move the automaton state onto a big root, so the bicoloured word
must be reduced. We leave the trivial verifications to the
enthusiastic reader.

So for every affine graph $G$ and every starting letter $s$
we have an infinite intervening neighbours words that is reduced.
With no further reference to Coxeter theory, it is now possible to
conclude that every intervening neighbours word in $G$ is reduced.
It is all a consequence of results about games on graphs.

\section{Roots and chips -- a polygon game }
\noindent
The construction of an intervening neighbours word $w$ has an
interpretation involving edge orientations. For each edge
\begin{picture}(23,9)(-2,0)
\put(3.5,3){\line(1,0){12}}
\put(-2,0){$s$}
\put(16.5,0){$t$}
\end{picture}
the neighbours $s$ and $t$ must alternate in $w$ and we may
indicate by the arrow
\begin{picture}(23,9)(-2,0)
\put(3.5,3){\vector(1,0){12}}
\put(-2,0){$s$}
\put(16.5,0){$t$}
\end{picture}
that $s$ was used last and that it is $t$s turn to occur next.
If all edges are directed in this way, the dynamics of a growing
intervening neighbours word can be formulated as a game with two rules:
\begin{itemize}
\item Choose any {\em sink} vertex $t$ and add its letter to the word.
\item Reverse all arrows into $t$ so that $t$ becomes a {\em source}.
\end{itemize}
This version of the game occurs in \cite{HK} but if each arrow-head is
detached and pronounced a chip, we have a special case of the
{\em chip-firing game} by Bj{\"o}rner, Lov\'asz and Shor \cite{BLS}.

So now an intervening neighbours word $w$ defines a move sequence in two
different games on the same graph and we are going to merge them 
into the {\em roots and chips game}. A game position is a nonnegative 
number on each vertex and an orientation of each edge. Initially there
are zeros on all vertices and the edge orientation
\begin{picture}(23,9)(-2,0)
\put(3.5,3){\vector(1,0){12}}
\put(-2,0){$t$}
\put(16.5,0){$u$}
\end{picture}
if $u$ occurs before $t$ in $w$. If the first letter in $w$ is $s$, 
we put a 1 on the $s$-vertex and reverse all its edge directions.
The subsequent moves combine the action of the automaton and the
chip-firing. Negative roots are not allowed, so  
\begin{picture}(39,9)(-2,0)
\put(-2.2,0){0}
\put(3.5,3){\vector(1,0){11}}
\put(14.5,0){1}
\put(30.3,3){\vector(-1,0){11}}
\put(32.6,0){0}
\end{picture}
is a terminal position. The sink in the middle cannot be played as
the value would change to $-1$.

The roots and chips game is a
{\em polygon game} in the terminology of K. Eriksson \cite{K}. In this
game, the polygons are diamonds. An example in $\tilde G_2$ follows.

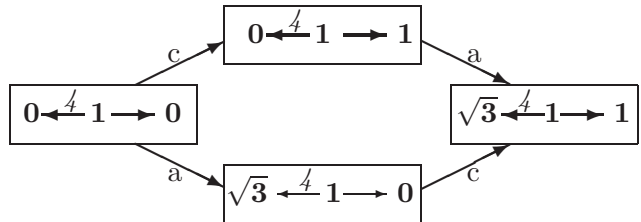
\hspace{27mm}
\begin{picture}(150,80)(40,-30)
    \thicklines
    \put(-5,0){\bf 0}
    \put(18,3){\vector(-1,0){16}}
    \put(10,6){\em\small 4}
    \put(20,0){\bf 1}
    \put(28,3){\vector(1,0){16}}
    \put(48,0){\bf 0}
    \thinlines
    \put(-10,-8){\line(1,0){70}}
    \put(-10,14){\line(1,0){70}}
    \put(-10,-8){\line(0,1){22}}
    \put( 60,-8){\line(0,1){22}}
    \thicklines
    \put(79,30){\bf 0}
    \put(102,33){\vector(-1,0){16}}
    \put(94,36){\em\small 4}
    \put(104,30){\bf 1}
    \put(115,33){\vector(1,0){16}}
    \put(135,30){\bf 1}
    \thinlines
    \put( 70,22){\line(1,0){74}}
    \put( 70,44){\line(1,0){74}}
    \put( 70,22){\line(0,1){22}}
    \put(144,22){\line(0,1){22}}
    \put(71,-30){$\bf\sqrt{3}$}
    \put(106,-27){\vector(-1,0){16}}
    \put(98,-24){\em\small 4}
    \put(108,-30){\bf 1}
    \put(115,-27){\vector(1,0){16}}
    \put(135,-30){\bf 0}
    \thinlines
    \put( 70,-38){\line(1,0){74}}
    \put( 70,-16){\line(1,0){74}}
    \put( 70,-38){\line(0,1){22}}
    \put(144,-38){\line(0,1){22}}
    \thicklines
    \put(37,14){\vector(2,1){33}}
    \put(49,23){c}
    \put(37,-8){\vector(2,-1){33}}
    \put(49,-22){a}
    \put(144,31){\vector(2,-1){33}}
    \put(161,23){a}
    \put(144,-25){\vector(2,1){33}}
    \put(161,-22){c}
    \thicklines
    \put(157,0){$\bf \sqrt{3}$}
    \put(190,3){\vector(-1,0){16}}
    \put(180,6){\em\small 4}
    \put(190,0){\bf 1}
    \put(196,3){\vector(1,0){16}}
    \put(216,0){\bf 1}
    \thinlines
    \put(155,-8){\line(1,0){70}}
    \put(155,14){\line(1,0){70}}
    \put(155,-8){\line(0,1){22}}
    \put(225,-8){\line(0,1){22}}
\end{picture} 

\medskip\noindent
Two neighbour vertices cannot both be sinks, so if two different 
moves are legal, they involve two non-neighbours. But then one
move does not influence neither edge directions nor root values
pertinent for the other move. Therefore the other move is still
legal and the result of making both moves is independent of the order.
And that defines a polygon game.

For polygon games, either all move sequences end in the same state after a finite
number of moves or all move sequences can be continued indefinitely.
For each of the affine groups we know 
one infinite move sequence, the bicoloured word, and we conclude that
every intervening neighbours word with the same initial state is
irreducible. The initial state includes the bicoloured edge orientation
which is only one of the $2^n$ possible edge orientations. So we
need to extend the result to all the others. Our last lemma shows
how this may be done, thus completing the proof of the main theorem.

\begin{lemma}
For each of the eight treelike affine graphs, there is an infinite 
roots and chips game that passes through all possible edge orientations. 
\end{lemma}
\begin{proof}
On a tree, any edge orientation can be transformed into any other edge
orientation by a chip-firing game. For the simple proof by induction,
see \cite{HK}. 

Starting with the bicoloured edge orientation, we can construct a game
that passes through all $2^n$ possible edge orientations, and we just
saw that the game will be infinite.\end{proof}

\end{document}